\documentclass[12pt]{amsart}
\usepackage{amsmath, amsthm, amssymb}
\usepackage[top=1.25in, bottom=1.25in, left=1.0in, right=1.0in]{geometry}

\allowdisplaybreaks
\usepackage{tkz-graph}
\usetikzlibrary{arrows}
\usetikzlibrary{shapes}
\usepackage[position=bottom]{subfig}

\makeatletter
\newtheorem*{rep@theorem}{\rep@title}
\newcommand{\newreptheorem}[2]{
\newenvironment{rep#1}[1]{
 \def\rep@title{#2 \ref{##1}}
 \begin{rep@theorem}}
 {\end{rep@theorem}}}
\makeatother

\theoremstyle{plain}
\newtheorem{thm}{Theorem}
\newreptheorem{thm}{Theorem}

\newreptheorem{prop}{Proposition}

\newreptheorem{lem}{Lemma}

\newreptheorem{conjecture}{Conjecture}

\newreptheorem{cor}{Corollary}

\theoremstyle{definition}

\theoremstyle{remark}

\newcommand{\card}[1]{\left|#1\right|}

\newcommand{\DefinedAs}{\mathrel{\mathop:}=}

\title{Yet another proof of Brooks' theorem}
\author{Landon Rabern}

\begin{document}
\maketitle

A referee for \emph{Brooks' Theorem and Beyond} \cite{cranston2014brooks} asked to make this proof available on the arXiv.  This is arguably the simplest variation as we avoid reducing to the cubic case entirely.

\begin{thm}[Brooks 1941 \cite{brooks1941colouring}]
Every graph $G$ with $\chi(G) = \Delta(G) + 1 \ge 4$ contains $K_{\Delta(G) + 1}$.
\end{thm}
\begin{proof}
Suppose the theorem is false and choose a counterexample $G$ minimizing
$\card{G}$.  Put $\Delta \DefinedAs \Delta(G)$. Using minimality of $\card{G}$,
we see that $\chi(G - v) \le \Delta$ for all $v \in
V(G)$. In particular, $G$ is $\Delta$-regular.

Let $M$ be a maximal independent set in $G$.  Since $\Delta(G-M) < \Delta$ and $\chi(G-M) \ge \Delta$, minimality of $|G|$ shows that $G-M$ has an induced subgraph $T$ where $T = K_\Delta$ or $T$ is an odd cycle if $\Delta=3$.   Suppose $G$ contains $K_{\Delta + 1}$ less an edge, say $K_{\Delta + 1} - xy = D \subseteq G$. Then we may $\Delta$-color $G-D$ and extend the coloring to $D$ by first coloring $x$ and $y$ the same and then finishing greedily on the rest.

Since $K_{\Delta + 1} \not \subseteq G$ we have $\card{N(T)} \ge 2$. So, we may take different $x, y \in N(T)$ and put $H \DefinedAs G - T$ if $x$ is adjacent to $y$ and $H \DefinedAs (G-T) + xy$ otherwise.  Then, $H$ doesn't contain $K_{\Delta + 1}$ as $G$ doesn't contain $K_{\Delta + 1}$ less an edge. By minimality of $\card{G}$, $H$ is $\Delta$-colorable. That is, we have a $\Delta$-coloring of $G - T$ where $x$ and $y$ receive different colors.  We can easily extend this partial coloring to all of $G$ since each vertex of $T$ has a set of $\Delta - 1$ available
colors and some pair of vertices in $T$ get different sets.  
\end{proof}

\bibliographystyle{plain}
\bibliography{GraphColoring}
\end{document}